\documentclass[11pt]{amsart}
\usepackage{amssymb}
\usepackage{amsmath,amssymb,amsfonts,amsthm,
latexsym, amscd, epsfig, epic,color}
\usepackage{mathrsfs}
\usepackage{eucal}
\usepackage{xspace}
\usepackage{a4wide}
\usepackage{colordvi}
\usepackage{comment}

\usepackage{xcolor}
\usepackage[pdfpagemode=UseNone,bookmarksopen=false,colorlinks=true,urlcolor=blue,citecolor=blue,citebordercolor=blue,linkcolor=blue]{hyperref}

\usepackage[top=2cm,left=2.5cm,right=2.5cm,bottom=2cm]{geometry}

\theoremstyle{plain}
\newtheorem{theorem}{Theorem}[section]

\newtheorem{lemma}{Lemma}[section]

\theoremstyle{definition}

\newtheorem{example}{Example}[section]

\theoremstyle{remark}

\newcommand{\ndR}{\mathbb{R}}

\newcommand{\sB}{\mathsf{B}}

\newcommand{\sG}{\mathsf{G}}

\newcommand{\sU}{\mathsf{U}}

\newcommand{\me}{\mathsf{e}}

\newcommand{\CRT}{\mathcal{T}_\me}

\newcommand{\convdis}{\,{\buildrel d \over \longrightarrow}\,}

\newcommand{\eqdist}{\,{\buildrel d \over =}\,}

\newcommand{\Di}{\textsc{D}}

\newcommand{\cK}{\mathcal{K}}

\newcommand{\cM}{\mathcal{M}}

\newcommand{\cR}{\mathcal{R}}
\newcommand{\cS}{\mathcal{S}}
\newcommand{\cT}{\mathcal{T}}
\newcommand{\cU}{\mathcal{U}}
\newcommand{\cV}{\mathcal{V}}

\renewcommand{\Pr}[1]{\mathbb{P}(#1)}

\newcommand{\Ex}[1]{\mathbb{E}[#1]}

\numberwithin{equation}{section}

\begin{document}
\title{Graph limits of random unlabelled $k$-trees}

\author{Emma Yu Jin and Benedikt Stufler}
\thanks{The first author was supported by the German Research Foundation DFG, JI 207/1-1, the Austrian Research Fund FWF, Project SFB F50-02/03, and is supported by FWF-MOST (Austria-Taiwan) project P2309-N35. The second author gratefully acknowledges support by the German Research Foundation DFG, STU 679/1-1 and the Swiss National Science Foundation grant number 200020\_172515.}
\address{Institut f\"{u}r Diskrete Mathematik und Geometrie, Technische Universit\"{a}t Wien, Wiedner Hauptstr. 8–10, 1040 Vienna, Austria}
\email{yu.jin@tuwien.ac.at}
\address{Institut f\"ur Mathematik, Universit\"at Z\"urich, Winterthurerstrasse 190, CH-8057 Z\"urich}
\email{benedikt.stufler@math.uzh.ch}

\begin{abstract}
We study random unlabelled $k$-dimensional trees by combining the colouring approach by Gainer-Dewar and Gessel (2014) with the cycle pointing method by Bodirsky, Fusy, Kang and Vigerske (2011). Our main applications are Gromov--Hausdorff--Prokhorov and  Benjamini--Schramm limits, that describe their asymptotic geometric shape on a global and local scale as the number of hedra tends to infinity.
\end{abstract}

\maketitle


\section{Introduction and main results}

	A {\em $k$-tree}, or \emph{$k$-dimensional tree}, may be defined recursively: it is either a complete graph on $k$
	vertices or a graph obtained from a smaller $k$-tree by adjoining a
	new vertex together with $k$ edges, connecting it to a $k$-clique of
	the smaller $k$-tree. This concept generalizes in a natural way graph-theoretic trees, which correspond to the special case $k=1$. We may distinguish $k$-trees whose vertices are labelled by elements of some fixed set, and \emph{unlabelled} $k$-trees, which are $k$-trees considered up to graph isomorphism. It is custom to index $k$-trees by their number of $(k+1)$-cliques, that are called \emph{hedra} in this context.  Thus, the number of vertices in a $k$-tree having $n$ hedra is given by $n+k$. For instance, there are $5$ different $2$-trees with $4$ hedra; see Figure~\ref{F:1}. A $k$-clique in a $k$-tree is usually called a \emph{front}. 
	\begin{figure}[htbp]
		\begin{center}
			\includegraphics[scale=0.8]{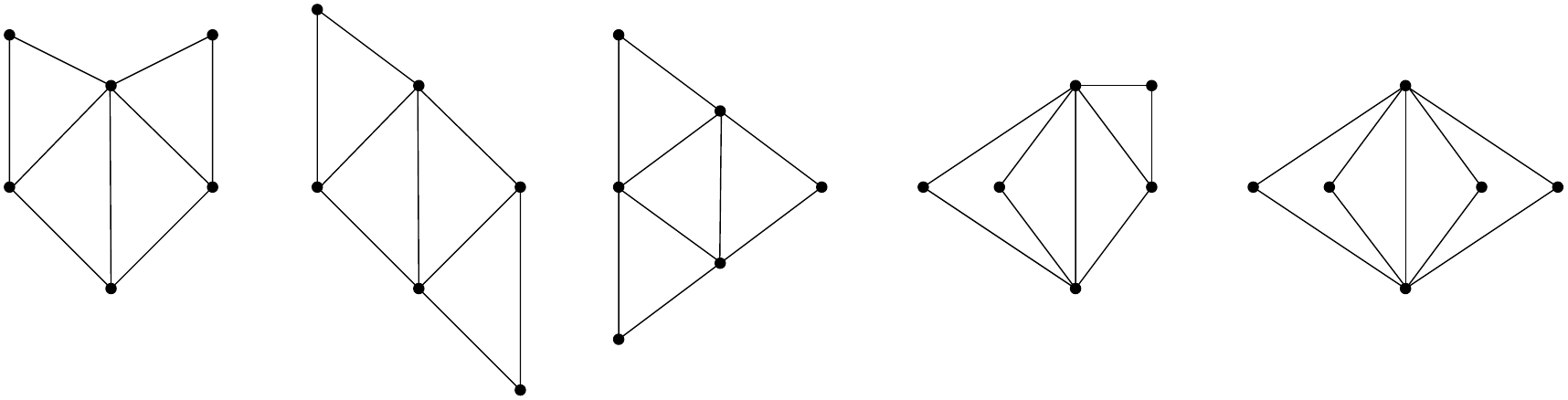}
			\caption{All unlabelled $2$-trees with $4$ hedra (triangles)\label{F:1}}
		\end{center}
	\end{figure}

	 The counting problem of the class of $k$-dimensional trees has a long history.  The number of labelled $k$-trees over a fixed set of vertices was obtained by Beineke, Pippert, Moon and Foata \cite{BP:69,Moon,Foata}, and the enumeration of unlabelled $1$-trees is a classical result attributed to Otter~\cite{Ot}. Unlabelled $2$-trees were counted by Harary and Palmer \cite{HP:68,HP:73} and Fowler {\it et al.}
	\cite{Fowler:02} using the dissimilarity characteristic theorem.
	The general case was a long-standing open problem, which was solved
	recently by Gainer-Dewar \cite{Gai} using $\Gamma$-species. A simpler proof that combines vertex-colourings with hedra-labelings was later discovered by  Gessel
	and Gainer-Dewar~\cite{GeGai}. The advantage of this approach is that it breaks the symmetry of $k$-trees and avoids the use of compatible cyclic orientation of each $(k+1)$-clique in a $k$-tree. Based on the simplified generating functions from \cite{GeGai}, Drmota and J. \cite{DJ:14} provided a systematic asymptotic
	analysis of $k$-trees using singularity analysis.
	
	In the present work we establish a substraction-free combinatorial decomposition of unlabelled $k$-dimensional trees.
	The motivation for this comes from the fact that all prior analysis of unlabelled $k$-trees are in one form or another based on dissymmetry theorems. These constitute double-counting arguments in terms of various rooted-versions of such objects. The substraction operations in the associated equations of generating series severly complicate a probabilistic analysis, as the corresponding Boltzmann sampling procedures have to employ a costly rejection process.  In order to tackle this, we combine the framework by Gessel and Gainer-Dewar~\cite{GeGai} and the cycle-pointed method developed by Bodirsky, Fusy, Kang and Vigerske \cite{BFKV:07}. The latter approach is based on the idea to consider graphs marked at a cyclic permutation of their vertices, such that the cycle appears in at least one automorphism. 
	
	Having a rejection-free sampling procedure at hand, we conduct a probabilistic study of 
	the random $k$-tree $\sU_n = \sU_{n,k}$ chosen uniformly at random among all unlabelled $k$-trees with $n$ hedra. A similar approach was also used in \cite{2014arXiv1412.6333S} for unlabelled trees with vertex-degree restrictions, and the present work intersects with this paper precisely for the case of unrestricted $1$-trees. The framework of the present work is not suitable to accomodate vertex-degree restrictions of $1$-trees and we make use of results for $\cR$-enriched trees~\cite{SENR}. The decomposition of~\cite{2014arXiv1412.6333S} is not suitable for $k$-trees if $k \ge 2$.
	
	Our first application establishes the Brownian tree $(\CRT, d_{\CRT}, \mu_{\CRT})$ as Gromov--Hausdorff--Prokhorov scaling limit of the random unlabelled $k$-tree $\sU_n$.
\begin{theorem}
	\label{te:ghlimit} 
	\label{te:A}
	Let $\mu_n$ denote the uniform measure on the set of vertices of $\sU_n$. There is a constant $c_k>0$ such that
	\begin{align*}
		(\sU_n, c_k n^{-1/2} d_{\sU_n}, \mu_n) \convdis (\CRT, d_{\CRT}, \mu_{\CRT})
	\end{align*}
	in the Gromov--Hausdorff--Prokhorov sense.
\end{theorem}
We refer the reader to \cite[Sec. 6]{Mi} for details on scaling limits of random graphs. The scaling constant is given by	
\begin{align}
	\label{eq:cc}
c_k = k \sum_{i=1}^k \frac{1}{i}\sqrt{1 + k \sum_{i=2}^\infty \bar{B}_{1^k}'(\rho_k^i)\rho_k^i} 
\end{align}
with $\bar{B}_{1^k}(z)$ the unique power-series satisfying
$
	\bar{B}_{1^k}(z) = z \exp\left(k \sum_{i=1}^\infty \frac{\bar{B}_{1^k}(z^i)}{i}\right),
$ and $\rho_k$ denoting its radius of convergence. See Table~\ref{tb:constants} for numerical approximations. It follows from~\cite[Thm. 3]{DJ:14} that $\rho_k = \frac{1}{ek} - \frac{1}{2e^3 k^2} + O(\frac{1}{k^3})$ and $k \sum_{i=2}^\infty \bar{B}_{1^k}'(\rho_k^i)\rho_k^i = O(\frac{1}{k})$ as $k$ becomes large, yielding \[c_k  = (1 + O(k^{-1}))k  \sum_{i=1}^k \frac{1}{i} .
\]

{
	\begin{table}[t]
		\begin{center}
			\tiny
			\begin{tabular}{ | l | l | l | l |  }
				\hline
				k &$c_k$&$\rho_k$& $\sqrt{1 + k \sum_{i=2}^\infty \bar{B}_{1^k}'(\rho_k^i)\rho_k^i}$\\
				\hline
				1 &  $1.102725$&    $0.338321$&  $1.102725$ \\
				2 &  $3.126190$ &   $0.177099$&  $1.042063$\\
				3 &  $5.643857$&    $0.119674$ & $1.026155$ \\
				4 &  $8.491071$ &   $0.090334$ & $1.018928$\\
				5 & $11.585821$ &   $0.072539$ & $1.014816$ \\
				6 & $14.878854$ &   $0.060597$ & $1.012166$\\
				7 & $18.337291$ &   $0.052031$ & $1.010319$ \\
				8 & $21.937615$ &   $0.045585$ & $1.008957$ \\		
				9 & $25.662173$ &   $0.040561$ & $1.007912$ \\				
				10 & $29.497218$ &  $0.036533$ & $1.007085$ \\						
				\hline
			\end{tabular}
		\end{center}
		\caption{Numerical approximations of constants for unlabelled $k$-trees}
		\label{tb:constants}
	\end{table}
}

The diameter $\Di(\cdot)$ is a Gromov--Hausdorff continuous functional. Hence  Theorem~\ref{te:ghlimit} implies that
\[
	c_k n^{-1/2} \Di(\sU_n) \convdis \Di(\CRT) \eqdist \sup_{0 \le t_1 \le t_2 \le 1}(\me(t_1) + \me(t_2) - 2 \inf_{t_1 \le t \le t_2} \me(t)),
\]
with $\me = (\me_t)_{0 \le t \le 1}$ denoting Brownian excursion of length $1$, see  Aldous \cite[Ch. 3.4]{MR1166406}. Let $v^1$ and $v^2$ denote two independently and uniformly selected vertices. The Gromov--Hausdorff--Prokhorov convergence of Theorem~\ref{te:ghlimit} implies that
\begin{align}
\label{eq:ray}
2 c_k d_{\sU_n}(v^1, v^2)/ \sqrt{n} \convdis \text{Rayleigh}(1)
\end{align} for a Rayleigh-distributed limit, given by its probability density $x\exp(-x^2/2)$. In fact, it also implies a scaling limit for the vector of pairwise distances for any finite fixed number of uniformly and independently sampled vertices.  In order to deduce convergence of the moments, it is necessary to verify $p$-uniform integrability of the diameter for arbitrarily large $p$. This is ensured by the following sharp tail-bound.

\begin{theorem}
	\label{te:dtb} 
	\label{te:B}
	There are constants $C,c>0$ such that for all $n \ge 1$ and $x \ge 0$
	\[
		\Pr{\Di(\sU_n) \ge x} \le C \exp(-c x^2/n).
	\]
\end{theorem}
Thus, for any fixed integer $p \ge 1$ it follows that
\[
	\Ex{\Di(\sU_n)^p} \sim c_k^{-p} n^{p/2} \Ex{\Di(\CRT)^p} \quad \text{and} \quad \Ex{  d_{\sU_n}(v^1, v^2)^p } \sim  n^{p/2} 2^{-p/2} c_k^{-p}\Gamma(1 + p/2).
\]
The moments of the diameter are known and given by
\begin{align*}
\Ex{\Di(\CRT)} &= \frac{4}{3}\sqrt{\pi/2}, \quad \Ex{\Di(\CRT)^2} = \frac{2}{3}\left(1 + \frac{\pi^2}{3}\right), \quad \Ex{\Di(\CRT)^3} = 2 \sqrt{2\pi}, \\
\Ex{\Di(\CRT)^k} &= \frac{2^{k/2}}{3} k(k-1)(k-3) \Gamma(k/2)(\zeta(k-2) - \zeta(k)) \quad \text{for $k \ge 4$}.
\end{align*}
Here $\zeta$ refers to the Riemann's zeta function, and Gamma to Euler's gamma function. See \cite[Sec. 3.4]{MR1166406} and \cite[Sec. 1.1]{outer}.

The second main application is a  local weak limit for $\sU_n$ that describes the asymptotic behaviour of the $r$-neighbourhoods $U_r(\sU_n, v^*)$ of a uniformly at random selected vertex $v^*$ of the graph $\sU_n$. We even obtain total variational convergence of these neighbourhoods when $r = r_n$ depends on $n$ and satisfies $r_n = o(\sqrt{n})$.

\begin{theorem}
	\label{te:bslimit} 
	\label{te:C}
	The random unlabelled $k$-tree $\sU_n$ converges in the Benjamini--Schramm sense towards a random infinite $k$-tree $\hat{\sU}$. For each sequence $r_n = o(\sqrt{n})$ it holds that
	\begin{align*}
		d_{\textsc{TV}}(U_{r_n}(\sU_n, v_n), U_{r_n}(\hat{\sU})) \to 0,
	\end{align*}
	with $v_n$ denoting a uniformly selected vertex of $\sU_n$.
\end{theorem}

This strengthened form of convergence is best-possible. Theorem~\ref{te:ghlimit} asserts that the diameter of the random unlabelled $k$-tree $\sU_n$ has order $\sqrt{n}$. Since  the diameter of $\hat{\sU}$ is almost surely infinite, the local convergence of $\sU_n$ towards $\hat{\sU}$ fails for $r_n$-neighbourhoods if $n^{-1/2}r_n$ does not converge to zero.

The cycle pointing approach allows us also to recover the expression for the asymptotic number of unlabelled $k$-trees with $n$ hedra obtained by Drmota and J.~\cite[Thm. 3]{DJ:14}. See Section~\ref{sec:conclusion} below for details.

It is important to keep in mind that in the present work we treat~\emph{unlabelled} $k$-trees, whose study is severely complicated by the presence of symmetries. Our results parallel a list of properties of random labelled $k$-trees, but do not encompass them and are not encompassed by them. The Rayleigh distribution has been observed to arise as limit of the distance of independent random vertices in random labelled $k$-trees by Darrasse and Soria~\cite{Dar}, but the scaling constant of \eqref{eq:ray} differs from the labelled case. Drmota, J., and S.~\cite{DJS:16} gave a scaling limit for random labelled $k$-trees, of course also with a different scaling constant, and S.~\cite{Streelike} established a Benjamini--Schramm limit that describes the asymptotic behaviour of the vicinity of a typical vertex in random labelled $k$-trees.

\section*{Notation}
Throughout, we set $[n] = \{1, 2, \ldots, n\}$ for all integers $n \ge 0$. The random variables appearing in this paper are either canonical or defined on a common probability space whose measure we denote by $\mathbb{P}$.  All unspecified limits are taken as $n$ becomes large.  We let $\convdis$ denote convergence in distribution, and denote equality in distribution by $\eqdist$. The total variation distance of measures and random variables is denoted by $d_{\textsc{TV}}$.  An event (that depends on $n$) holds with high probability, if its probability tends to $1$ as $n$ tends to infinity. We say it is \emph{exponentially unlikely} if there are constants $C,c>0$ such that its probability is bounded by $C \exp(-cn)$ for all $n$. Likewise, we say it is \emph{exponentially likely} if its complement is exponentially unlikely. For any integer $n \ge 0$ and any power series $f(z)$ we let $[z^n]f(z)$ denote the coefficient of $z^n$ in $f(z)$.





\section{Gainer-Dewar's and Gessel's decomposition}\label{S:enum1}

\subsection{Vertex colourings, hedron labelings, and a bijection with coding trees}
\label{sec:bij}
We recall some results and terminology from \cite{Gai,GeGai}.
Any two hedra $h_1$ and $h_2$ in a $k$-tree that intersect at a front $f$ are termed \emph{adjacent}. If this is the case, then a front $f_1$ of $h_1$ and a front $f_2$ of $h_2$ are called \emph{mirror with respect to $f$} if $f_1 \cap f = f_2 \cap f$.

A {\em coloured hedron-labelled $k$-tree} with $n$ hedra is a $k$-tree where the hedra are labelled by distinct integers from $[n]$ and the fronts are coloured with integers from $[k+1]$. We require that any two distinct fronts  that belong to the same hedron must have distinct colours, and any two distinct fronts that are mirror with respect to some other front must have the same colour. This way, the $k+1$ fronts belonging to any single hedron are coloured with distinct integers from $[k+1]$. See Figure~\ref{F:2} for two examples in the special case $k=2$, where labels are denoted by boxed integers.

\begin{figure}[t]
	\begin{center}
		\includegraphics[scale=1.0]{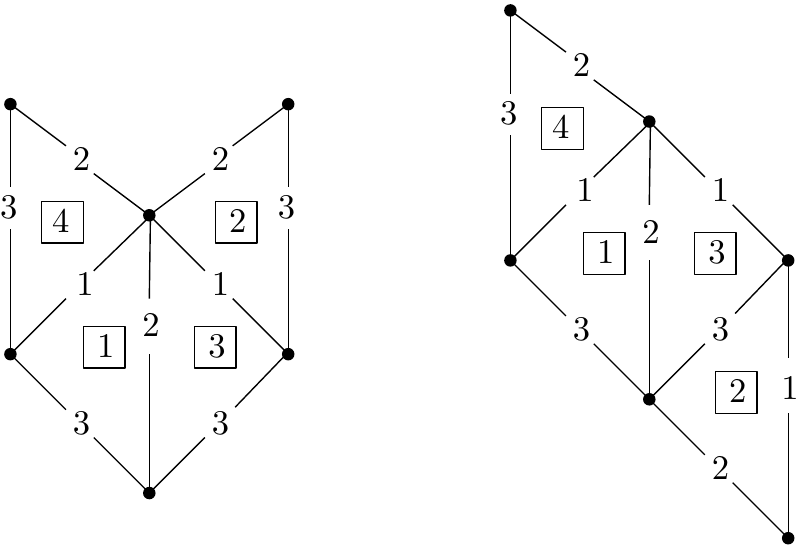}
		\caption{Two coloured hedron-labelled $2$-trees with $4$ hedra\label{F:2}}
	\end{center}
\end{figure}

It is not hard to see that the colours of all fronts of any single hedron already determine the colours of all other fronts in the $k$-tree. However, the total number of front-colourings may vary according to the $k$-tree we consider.

We now introduce {\em $k$-coding trees}. A (coloured and labelled) $k$-coding tree is an unordered tree with a proper bipartition of its vertex set into white and black vertices. We require that each black vertex has precisely $k+1$ white neighbours. The $n$ black vertices are labelled with distinct integers from $[n]$, and to each white-vertex we assign a colour from $[k+1]$, such that each black vertex has precisely one neighbour with colour $i$ for all $i \in [k+1]$.

There is a bijection $\phi$ between the set $\cK_{n,k}$ of coloured hedron-labelled $k$-trees with $n$ hedra, and the set $\cT_{n,k}$ of (coloured and labelled) coding trees with $n$ black vertices. The  proof is analogous to \cite[Thm. 3.4]{Gai}:

To construct a $k$-coding tree from a coloured hedron-labelled $k$-tree, we assign to each hedron a black-vertex with the same label and to each front a white vertex with the same colour. We connect a white vertex with a black vertex by an edge if the front corresponding to the white vertex is a subset of the hedron corresponding to the black vertex.

For the inverse map, note that in order to construct a $k$-tree from a coding tree we  require knowledge of the colouring. There are multiple ways to glue two front-coloured hedra together at a specified front colour, but only one way such that afterwards any pair for fronts that are mirror with respect to the resulting shared front have the same colour. See for example Figure~\ref{F:3} for the $2$-coding trees that correspond to the  $2$-trees in Figure~\ref{F:2}.

For any integer $n \ge 0$ we let $\mathfrak{S}_n$ denote the the symmetric group of degree $n$. The  groups $\mathfrak{S}_n$ and $\mathfrak{S}_{k+1}$ both operate on  the set $\cK_{n,k}$ of coloured and labelled $k$-trees, and the two actions commute. This induces an action of the group $\mathfrak{S}_n$ on the set of orbits $\cK_{n,k} / \mathfrak{S}_{k+1}$, that may be identified with $k$-trees on unlabelled vertices with labels on the hedra.

Any graph isomorphism between $k$-trees also induces a bijection between their sets of hedra. Thus, any two hedron-labelled $k$-trees are identical as unlabelled graphs if and only if one may obtained from the other via relabelling of hedra. Thus, the $\mathfrak{S}_n$-orbits of the induced action correspond precisely to the unlabelled $k$-trees with $n$ hedra. 

Since the two actions on $\cK_{n,k}$ commute, it follows that there is also a canonical correspondence between unlabelled $k$-trees and the collection of orbits of the induced group action of $\cS_{k+1}$ on the set of orbits $\cK_{n,k} / \mathfrak{S}_{n}$. Elements of $\cK_{n,k} / \mathfrak{S}_{n}$ correspond to $k$-trees that are unlabelled but coloured. We refer to the orbits of $\cS_{k+1}$ on this set as \emph{colour-orbits of unlabelled $k$-trees}.

As the bijection $\phi$ is compatible with the actions of both groups $\mathfrak{S}_n$ and $\mathfrak{S}_{k+1}$, this reduces the study of unlabelled $k$-trees to the study of colour-orbits of unlabelled $k$-coding trees, that is, orbits of the induced action of $\mathfrak{S}_{k+1}$ on the collection of orbits $\cT_{n,k} / \mathfrak{S}_{n}$.
\begin{figure}[t]
	\begin{center}
		\includegraphics[scale=1.0]{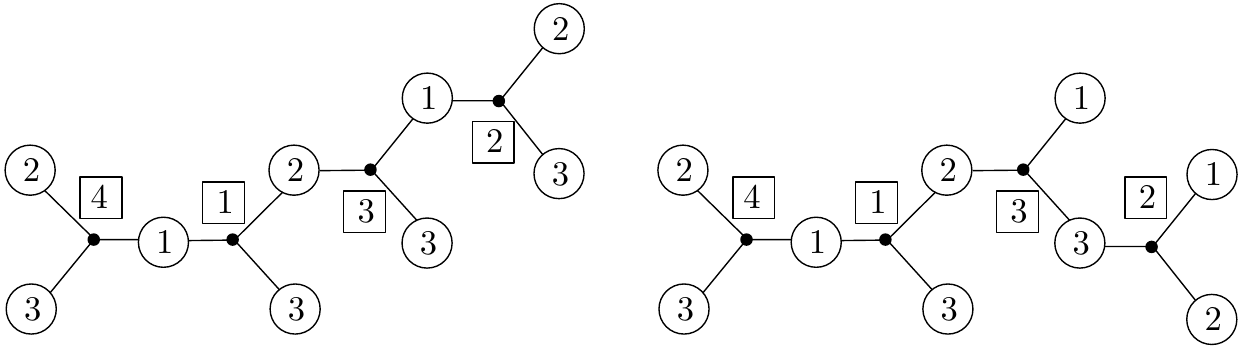}
		\caption{Two $2$-coding trees that correspond to the $2$-trees in Figure~\ref{F:2}.\label{F:3}}
	\end{center}
\end{figure}

\subsection{Burnside's Lemma}
\label{sec:burn}
The enumeration of colour-orbits of unlabelled $k$-trees and $k$-coding trees is undertaking using Burnside's Lemma, which we briefly recall in this section. Given a permutation $\sigma\in\mathfrak{S}_m$ its cycle type $\lambda = (\lambda_i)_{1 \le i \le m}$ 
is defined by letting $\lambda_i$ denote the number of cycles of length $i$ in the decomposition of $\sigma$ into a product of disjoint cycles. It is custom to use the formal notation
$
\lambda =(1^{\lambda_1}\,2^{\lambda_2}\,\cdots\, m^{\lambda_m}),
$ 
and we will often drop the parentheses when there is no risk of confusion.
Since  $m=\lambda_1+2\lambda_2+\cdots+m\lambda_m$ we say that $\lambda$ is a partition of $m$ and denote this by this by $\lambda \vdash m$. 
We set
\begin{align}\label{E:zlam}
z_\lambda = 1^{\lambda_1} \lambda_1! 2^{\lambda_2}
\lambda_2! \cdots m^{\lambda_m} \lambda_m!,
\end{align}
so that
$m!/z_\lambda$ is the number of permutations in
$\mathfrak{S}_m$ of cycle type $\lambda$. Moreover, for any $d \ge 1$ we let $\lambda^d$ denote the cycle type of the $d$-th power of a fixed permutation with type~$\lambda$.

	We let $A \subset \ndR[[z]]$ denote the subset of all formal power series whose coefficients are non-negative. Suppose that we are given a non-empty set $S$ together with a weight-function $\omega: S \to A$, such that the sum
	$
	\sum_{s \in S} \omega(s)
	$
	is well-defined in $A$. That is, for any $n \ge 0$ the coefficients $([z^n]\omega(s))_{s \in S}$ form a summable family of non-negative real numbers. Suppose that we are additionally given a group-action of the symmetric group $\mathfrak{S}_m$ on $S$ that preserves weights.  Thus, all elements of a common orbit $O$ have the same $\omega$-weight, which we denote by  $\omega(O)$ and call the weight of the orbit. For each permutation $\sigma \in \mathfrak{S}_m$ we let $\mathrm{Fix}(\sigma) = \{s \in S \mid \sigma.s = s\}$ denote the set of fixpoints of $\sigma$. The corresponding inventory $\mathrm{Fix}_\lambda := \sum_{s \in \mathrm{Fix}(\sigma)} \omega(s)$ only depends on the cycle type $\lambda \vdash m$ of $\sigma$.

\begin{lemma}[Burnside's lemma for the symmetric group]\label{L:burn}
	\label{le:burn}
	The sum of the weights of all $\mathfrak{S}_m$-orbits is given by
	$
	\sum_{O \in S / \mathfrak{S}_m}	\omega(O) = \sum_{\lambda\vdash m}\frac{\mathrm{Fix}_{\lambda}}{z_{\lambda}}.
	$
\end{lemma}

Suppose that for each type $\lambda \vdash m$ we fix some permutation $\sigma_\lambda \in \mathfrak{S}_m$ with type $\lambda$. Let $Z_m = \sum_{O \in S / \mathfrak{S}_m} \omega(O)$ denote the sum of the weights of all orbits. The following probabilistic application of Burnside's lemma will turn out useful.
\begin{lemma}
	\label{le:bprob}
	Suppose that all $\omega$-weights are positive real numbers. We may sample a random type $\overline{\lambda} \vdash m$ with probability
	$
	\Pr{\overline{\lambda} = \lambda} = \frac{\mathrm{Fix}_{\lambda}}{z_{\lambda}} Z_m^{-1}
	$
	and then select an element $\overline{s}$ from $\mathrm{Fix}(\sigma_{\overline{\lambda}})$ with probability proportional to its $\omega$-weight.
	Then the orbit $\overline{O}$ corresponding to $\overline{s}$ is distributed according to
	$
	\Pr{\overline{O} = O} = \omega(O) Z_m^{-1}.
	$
\end{lemma}
\begin{proof}
	Let $O \in S / \mathfrak{S}_m$ be an arbitrary orbit. Clearly the symmetric group $\mathfrak{S}_m$ also operates on $O$, and applying Burnside's Lemma~\ref{le:burn} to this operation yields
	\[
		\omega(O) = \sum_{\lambda \vdash m} z_\lambda^{-1} \sum_{s \in \mathrm{Fix}(\sigma_\lambda) \cap O} \omega(s).
	\]
	Thus
	\begin{align*}
		\Pr{\overline{O} = O}  = \sum_{\lambda \vdash m} \Pr{\overline{\lambda}=\lambda} \Pr{\overline{s} \in \mathrm{Fix}(\sigma_\lambda) \cap O \mid \overline{\lambda}=\lambda} 
		= Z_m^{-1} \sum_{\lambda \vdash m} z_\lambda^{-1} \sum_{s \in \mathrm{Fix}(\sigma_\lambda) \cap O} \omega(s) 
		= Z_m^{-1}\omega(O).
	\end{align*}
\end{proof}

The operation of the group $\mathfrak{S}_m$ on the set $S$ induces an operation on the set $M(S)$ of all finite multi-sets of elements in $S$. The weight-function $\omega$ on $S$ extends in a natural to $M(S)$ by defining the weight of a multi-set to be the product of the weights of its elements (with repetitions). For any $\sigma \in \mathfrak{S}_m$ we let $\mathrm{Fix}^{M(S)}(\sigma)$ denote the set of all $M \in M(S)$ satisfying~$\sigma.M = M$.

\begin{lemma}[{\cite[Lem. 2]{GeGai}}]
\label{le:mult}
For each $\sigma \in \mathfrak{S}_m$ it holds that
\begin{align}
	\label{eq:ts}
	\sum_{M \in \mathrm{Fix}^{M(S)}(\sigma)} \omega(M) = \exp\left( \sum_{i\ge1} \frac{1}{i} \sum_{s \in \mathrm{Fix}^{S}(\sigma^i)} \omega(s)^i \right).
\end{align}
\end{lemma}

In {\cite[Lem. 2]{GeGai}} such a result was stated, however instead of taking the power $\omega(s)^i$ on the right-hand side, a substitution operation $\omega(s)(z^i)$ was employed. This makes no difference for the cases in which this result is applied here or in \cite{GeGai}, because then $\omega(s)$ is always some power of $z$ and the two operations coincide, but just to be sure we verify Lemma~\ref{le:mult}:

\begin{proof}[Proof of Lemma~\ref{le:mult}]
	A multiset $M \in M(S)$ is fixed by $\sigma$, if and only if it is a multi-set union of orbits of the operation of the generated subgroup $<\sigma>$ on the set $S$. So let $(O_j)_{j \in J}$ denote the collection of these orbits. For each $j \in J$ we set $r_j = |O_j|$ and select a representative $s_j \in O_j$. Any $M \in M(S)$ may uniquely be written as the multi-set union of $\ell_j \ge 0$ copies of $O_j$ for all $j \in J$, with $\sum_{j \in J} \ell_j < \infty$, it follows that
	\begin{align}
		\label{eq:p}
		\sum_{M \in \mathrm{Fix}^{M(S)}(\sigma)} \omega(M) = \prod_{j \in J} \sum_{\ell \ge 0} \omega(s_j)^{\ell r_j} = \prod_{j \in J} \frac{1}{1 - \omega(s_j)^{r_j}}
	\end{align}
	Here we have used the assumption, that the family $(\omega(s))_{s \in S}$ is summable, which implies that all products with infinitely many factors $\ne 1$ in Equation~\eqref{eq:p} vanish. That is, we really only sum up weights of finite multi-sets. Applying the logarithm operator to Equation~\eqref{eq:p} yields 
	\begin{align}
		\label{eq:refme}
		\log \left( \sum_{M \in \mathrm{Fix}^{M(S)}(\sigma)} \omega(M) \right) = \sum_{j \in J} \sum_{\ell \ge 1} \frac{\omega(s_j)^{\ell r_j}}{\ell}.
	\end{align}
	
	We now focus on the argument of the exponential operator on the right-hand side of Equation~\eqref{eq:ts}.  Clearly we may write
	\begin{align}
		\sum_{i=1}^\infty \frac{1}{i} \sum_{s \in \mathrm{Fix}^{S}(\sigma^i)} \omega(s)^i = \sum_{j \in J} \sum_{i\ge1} \frac{1}{i} \sum_{s \in O_j \cap \mathrm{Fix}^{S}(\sigma^i)} \omega(s)^i.
	\end{align}
	It holds that $O_j \subset \mathrm{Fix}^{S}(\sigma^i)$ if $i$ is a  multiple of $r_j$, and $O_j \cap \mathrm{Fix}^{S}(\sigma^i) = \emptyset$ otherwise. Hence
	\begin{align}
	\sum_{j \in J} \sum_{i\ge1} \frac{1}{i} \sum_{s \in O_j \cap \mathrm{Fix}^{S}(\sigma^i)} \omega(s)^i = \sum_{j \in J} \sum_{\ell \ge 1} \frac{\omega(s)^{\ell r_j}}{\ell}.
	\end{align}
	Together with Equation~\eqref{eq:refme}, this verifies Equation~\eqref{eq:ts}.
\end{proof}

\subsection{Generating functions}
\label{sec:genfun}

We let $U(z)$ denote the generating series of unlabelled $k$-trees indexed by their number of hedra. Equivalently, we may state that $U(z)$ is the generating series of colour-orbits of unlabelled $k$-coding trees, indexed by their number of black vertices. The dissymmetry theorem proved by Gainer-Dewar and Gessel~\cite[Lem. 5, 6]{GeGai} expresses this function by the Equation
\begin{align}
\label{E:dissym}
U(z)=B(z)+C(z)-E(z).
\end{align}
Here $B(z)$, $C(z)$, and $E(z)$ denote the generating functions for colour-orbits of unlabelled $k$-coding trees that are rooted at a black vertex, a white vertex, and an edge, respectively. That is, the coefficient of $z^n$ in these series is formed by counting the number of orbits of the action of $\mathfrak{S}_{k+1}$ on the collection of $\mathfrak{S}_n$-orbits corresponding to the relabeling operation on the set of all pairs $(T,v)$ of a labelled and coloured $k$-coding tree $T$ having $n$ black vertices, and a root $v$ which is a marked black vertex, white vertex, or edge.

Our  goal in Section~\ref{sec:cyc} below is to provide a substraction-free alternative to Equation~\eqref{E:dissym}. We are going to build on the results of Gainer-Dewar and Gessel \cite[Thm. 7]{GeGai} concerning $k$-coding trees rooted at a black or white vertex. These classes may be decomposed by applying Lemma~\ref{le:mult} and Burnside's lemma (Lemma~\ref{le:burn}) to recolouring operations on marked, unlabelled, and coloured objects. We briefly recall the arguments, as we are going to use these decompositions (rather than just the resulting equations of generating functions) later on.
 
For any cycle type $\lambda  \vdash k+1$  we may fix a permutation $\pi_\lambda\in \mathfrak{S}_{k+1}$ having type $\lambda$  and let  $B_\lambda(z)$ denote the generating function for coloured, unlabelled, black-rooted trees that are invariant under re-colouring by $\pi_\lambda$. Furthermore, for any $i \ge 1$ we let $\lambda^i$ denote the cycle type of $\pi_\lambda^i$. This notion does not depend on the choice of permutation.   Burnside's lemma yields 
\begin{align}
\label{E:B1} B(z)&=\sum_{\lambda\vdash
	k+1}\frac{B_{\lambda}(z)}{z_{\lambda}}.
\end{align}

Each colour-orbit of a $C$-object contains a coloured, unlabelled coding-tree where the white root-vertex has colour $k+1$.  Thus the colour-orbits of the action of $\mathfrak{S}_{k+1}$ on all white-rooted, coloured, unlabelled coding trees correspond precisely to the colour-orbits of the action of $\mathfrak{S}_k$ on coloured, unlabelled coding trees marked at a white-vertex with colour~$k+1$. Applying Burnside's Lemma to this action of $\mathfrak{S}_k$ yields
\begin{align}
\label{E:C1} C(z)&= \sum_{\mu\vdash k} \frac{C_{\mu}(z)}{z_{\mu}}
\end{align}
with $C_\mu(z)$ denoting the generating series of all coloured, unlabelled $k$-coding trees that are rooted at a white vertex with colour $k+1$ and invariant under recolouring by a fixed (but arbitrary) permutation $\sigma_\mu \in \mathfrak{S}_k$ with type $\mu$.

We define the generating function $\bar{B}_\mu(z)$ in the same way as $C_\mu(z)$, but only count the trees where the white root with colour $k+1$ has precisely one black neighbor.  This black neighbor may be interpreted as a black root vertex, and $\bar{B}_\mu$-objects are termed black-rooted reduced trees. 

We may view a white-rooted, coloured, unlabelled $k$-coding tree whose root has colour $k+1$ as a  multi-set of such trees where additionally the white root has precisely one neighbour. Hence
Lemma~\ref{le:mult} applies, yielding
\begin{align}
\label{E:CexpB1}C_{\mu}(z)&=\exp
\left(\sum_{i=1}^{\infty}\frac{\bar{B}_{\mu^i}(z^i)}{i}\right).
\end{align}

If we delete the root of a $\bar{B}_\mu$-object, we are left with $k$ white-rooted unlabelled coloured $k$-coding trees $T_1, \ldots, T_k$ whose roots are coloured from $1$ to $k$. For any cycle $c = (c_1, \ldots, c_\ell)$ of $\sigma_\mu$ the trees $T_{c_1}, \ldots, T_{c_\ell}$ belong to the same colour-orbit, and each is invariant under relabeling by $\sigma_\mu^\ell$. Setting $d = \min(c_1, \ldots, c_\ell)$, the result of switching the colour $d$ with the colour $k+1$ in the tree $T_d$ yields a reduced tree, that together with the cycle $c$ already contains all information on the $T_1, \ldots, T_\ell$. Hence the trees corresponding to $c$ are enumerated by $C_{\mu^\ell}(z^\ell)$, and the generating series for $\bar{B}_\mu$-objects is given by
\begin{align}
\label{E:BbarC1}\bar{B}_{\mu}(z)&=z\prod_{i \in \mu} C_{\mu^i}(z^i),
\end{align}
with the index $i$ ranging over all parts of the type $\mu$. Similarly, we may argue that
\begin{align}
\label{E:BtoC}B_{\lambda}(z)&=z\prod_{i \in \lambda} C_{\lambda^i-(1,0,\ldots,0)}(z^i),
\end{align}
with $\lambda^i-(1,0,\ldots)$ denoting the cycle type obtained by removing one part of length $1$ from~$\lambda^i$.

\section{$k$-trees rooted at a front of distinguishable vertices}
\label{sec:disting}

 We let $\rho_k$ denote the radius of convergence of the generating series $U(z)$ of unlabelled $k$-trees.  Drmota and J.~\cite{DJ:14} established the following asymptotic enumerative result, showing the special role of the cycle type $1^k$ in this context.

\begin{lemma}[{\cite[Thm. 3]{DJ:14}}]
	\label{le:analytic}
	The series  $C_{1^k}(z)$, and $\bar{B}_{1^k}(z)$ have a dominant singularity of square-root type at~$\rho_k<1$ and it holds that $\bar{B}_{1^k}(\rho_k) = k^{-1}$.	The series $C_\mu(z)$ and $\bar{B}_\mu(z)$ are analytic at $\rho_k$ if $\mu \ne 1^k$ is not the cycle-type of the identity map. The series $U(z)$ has a dominant singularity of type $(1 - z/\rho_k)^{-3/2}$.
\end{lemma}

The class  of labelled $k$-trees admits a recursive decomposition~\cite{DJS:16} that is based on $k$-trees rooted at a front of distinguishable vertices.  Two such elements are considered isomorphic, if there is a graph isomorphism that \emph{pointwisely} preserves the root-front. Hence the corresponding cycle-index sums do not count front-rooted unlabelled $k$-trees, but unlabelled $k$-trees that are rooted at a front of \emph{distinguishable} vertices.

This relates to the present setting as follows.
The $k$-trees counted by $\bar{B}_{1^k}(z)$ are unlabelled and coloured, with a root-front of colour $k+1$ that is contained in a unique hedron. The colours $1$ to $k$ of the remaining fronts of this hedron uniquely determine the front-colouring of the entire $k$-tree, and may be interpreted as a labelling of the $k$ vertices of the root-front. That is,  $\bar{B}_{1^k}(z)$ counts unlabelled uncoloured $k$-trees that are rooted at a front of $k$ distinguishable vertices that is contained in a unique hedron. The series $C_{1^k}(z)$ counts such objects without the restraint of the root-front having to belong to a unique hedron. By \eqref{E:CexpB1} and \eqref{E:BbarC1} these series satisfy the equations
\begin{align}
\label{eq:recur}
\bar{B}_{1^k}(z) = z\exp
\left(k\sum_{i=1}^{\infty}\frac{\bar{B}_{1^k}(z^i)}{i}\right) \quad \text{and} \quad C_{1^k}(z) = \exp\left(\sum_{i=1}^{\infty}\frac{\bar{B}_{1^k}(z^i)}{i}\right),
\end{align}
which of course agree with the cycle-index sums associated to the decomposition of labelled $k$-trees in~\cite{DJS:16}.

In~\cite{SENR} $k$-trees rooted at a front of distinguishable vertices were studied as special cases of unlabelled $\cR$-enriched trees. Let $c_k$ be defined as in \eqref{eq:cc}.
\begin{lemma}[{\cite[Sec. 6.5]{SENR}}]
	\label{le:base}
	Let $\sG_n$ be either the uniform $n$-hedra $k$-tree from the class $\bar{B}_{1^k}$ or $C_{1^k}$. Let $\mu_n^{\sG}$ denote the uniform measure on the vertices of $\sG_n$.  The rescaled space $(\sG_n, c_k n^{-1/2} d_{\sG_n}, \mu_n^{\sG})$ converges in the Gromov--Hausdorff--Prokhorov sense towards the Brownian tree. There are constants $C,c>0$ such that $\Pr{\Di(\sG_n) \ge x} \le C \exp(-cx^2/n)$ for all $n$ and $x \ge 0$. Let $v_n$ be a vertex sampled according to $\mu_n$. There is an infinite rooted random graph $\hat{\sG}$ such that for any sequence $r_n = o(\sqrt{n})$ the $r_n$-neighbourhood $U_{r_n}(\cdot)$ satisfies $d_{\textsf{TV}}( U_{r_n}(\sG_n, v_n), U_{r_n}(\hat{\sG})) \to 0$.
\end{lemma}
Here the limit graph $\hat{\sG}$ does not depend on whether we consider random elements of the class $\bar{B}_{1^k}$ or of the class $C_{1^k}$. To be precise, \cite{SENR} established Gromov--Hausdorff convergence of $\sG_n$ to the Brownian tree,  but it is not hard to see that the arguments may be extended to obtain Gromov--Hausdorff--Prokhorov convergence. The scaling constant $c_k$ is explicit in~\cite[First display after Eq. (7.29)]{SENR} and may be seen to be identical to the expression in \eqref{eq:cc} by straight-forward calculations.

The numerical approximation of the constant $\rho_k$ in Table~\ref{tb:constants} was formed by taking $m:=30$, calculating a truncation $\bar{B}^{[m]}_{1^k}(z)$ up to order $m$ of $\bar{B}_{1^k}(z)$ by using the recursive relation~\eqref{eq:recur}, and numerically solving the truncated system $x \exp(k \sum_{i=2}^m i^{-1}\bar{B}^{[m]}_{1^k}(x^i)) = 1/(ek)$.

\section{A substraction-free decomposition}
\label{sec:cyc}


\subsection{Cycle pointing}

Let $m \ge 0$ be an integer. Recall that any permutation $\sigma \in \mathfrak{S}_m$ may be decomposed in a unique way into a product of disjoint cycles. The cycles  correspond to the orbits of the action of the generated subgroup $<\sigma> \subset \mathfrak{S}_m$ on the set $[m]$. Here we count fixed points as $1$-cycles. In the following, we say $c$ is a cycle of $\sigma$, if $c$ is one of the factors in this decomposition.

Suppose that the symmetric group  $\mathfrak{S}_m$ acts on a set $S$. We may consider the \emph{cycle pointed} set $S^\circ$ of all pairs $(s, c)$ of an element $s \in S$ together with a \emph{marked cycle} $c$, for which \emph{at least} one permutation $\sigma \in \mathfrak{S}_m$ exists that satisfies $\sigma.s = s$ and that has $c$ as one of its cycles. Naturally, the group $\mathfrak{S}_m$ acts on $S^\circ$ via $\nu.(s,c) = (\nu.s, \nu c \nu^{-1})$ for all $\nu \in \mathfrak{S}_m$ and $(s,c) \in S^\circ$. 

There is a well-defined map $S^\circ / \mathfrak{S}_m \to S / \mathfrak{S}_m$, that sends the orbit of an element $(s,c) \in S^\circ$ to the orbit of  $s$. By \cite[Thm. 15]{BFKV:07}, the pre-image of any orbit in $S / \mathfrak{S}_m$ has precisely $m$ elements. This completely reduces the task of counting $S / \mathfrak{S}_m$ to the task of counting~$S^\circ / \mathfrak{S}_m$. The latter may be easier, as the marked cycle provides a point of reference.

Recall that the groups $\mathfrak{S}_n$ and $\mathfrak{S}_{k+1}$  operate on the class $\cK_{n,k}$ of front-coloured and hedron-labelled $k$-trees, and that the two operations commute. We would like to study unlabelled, uncoloured $k$-trees, that correspond bijectively to the elements of the collection $(\cK_{n,k} / \mathfrak{S}_n) / \mathfrak{S}_{k+1}$ of colour-orbits of unlabelled, coloured $k$-trees.

We have to take great care when trying to apply the cycle-pointing method to this setting, as there are luring pitfalls: For example, we could cycle point the operation on $\mathfrak{S}_n$ on $\cK_{n,k}$, resulting in a set $\cK_{n,k}^{\circ_n}$. The orbits $\cK_{n,k}^{\circ_n} / \mathfrak{S}_n$ are in a $n$ to $1$ correspondence to the unlabelled coloured $k$-trees from $\cK_{n,k} / \mathfrak{S}_n$, however this relation breaks when passing to the colour-orbits. That is, the orbits $(\cK_{n,k}^{\circ_n} / \mathfrak{S}_n) / \mathfrak{S}_{k+1}$ satisfy, in general, no longer an~$n$ to~$1$ correspondence to the unlabelled, uncoloured trees $(\cK_{n,k} / \mathfrak{S}_n) / \mathfrak{S}_{k+1}$. A counter-example where this relation fails is already given for the special case $n=5$ and $k = 1$.

Let $\cU_n$ denote the set of all unlabelled, uncoloured $k$-trees with $n$ hedra. What we are going to do is consider the action of the symmetric group $\mathfrak{S}_n$ on the set $\cK_n$ of $n$-hedron $k$-trees with hedra labelled from $1$ to $n$. Clearly there is a bijection from $\cK_n /  \mathfrak{S}_n$ to $\cU_n$, and consequently an $n$ to $1$ correspondence from $\cV_n := \cK_n^\circ / \mathfrak{S}_n$ to $\cU_n$. Compare with Figure~\ref{fi:cyc}.

\begin{figure}[t]
	\label{fi:cyc}
	\begin{center}
		\includegraphics[scale=0.6]{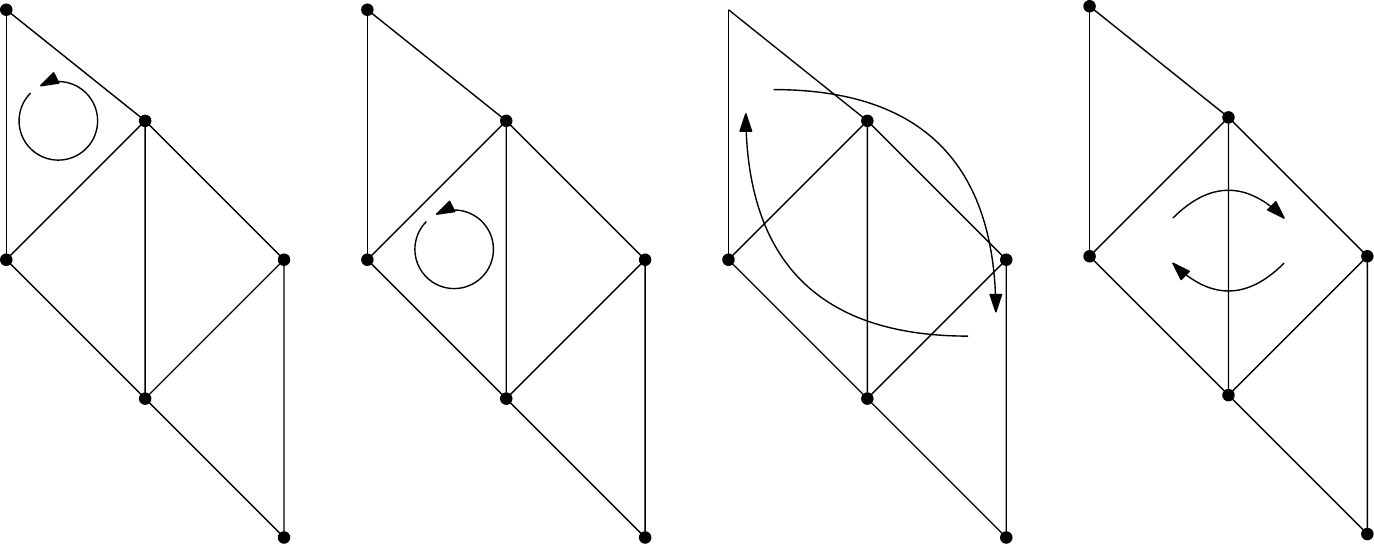}
		\caption{The four possible cycle-pointing of a $2$-tree with four hedra.}
	\end{center}
\end{figure}

Given a hedron-labelled (uncoloured) $k$-tree $K \in \cK_n$ we may consider the labelled tree $T$ whose vertices correspond to the hedra of $K$, with two vertices being incident if the corresponding hedra share a front.  Any permutation $\sigma \in \mathfrak{S}_n$ with $\sigma.K$ induces a graph-isomorphism on $K$ and hence a graph isomorphism on $T$. If we mark a cycle $c$ of $K$, then~$c$ is also a marked cycle of $T$. Thus, an unlabelled cycle-pointed $k$-tree from $\cV_n$ induces an unlabelled cycle-pointed tree.

There are three types of cycle-pointed trees, as was established in \cite[Prop. 24]{BFKV:07}. The first type are trees with a marked fix-point. If the marked cycle has length at least $\ell \ge 2$, then one may consider the corresponding $\ell$ \emph{connecting paths}, that join consecutive atoms of $c$. Each of these paths has a center, which may either be a vertex or an edge, and the centers of all connecting paths coincide \cite[Claim 22]{BFKV:07}. This allows us to distinguish between the second type, where this $\emph{cycle center}$ is a vertex, and the third type, where it is an edge. According to the three possible types of the cycle pointed (unlabelled, uncoloured) tree associated to a cycle pointed (unlabelled, uncoloured) $k$-tree, we may split $\cV_n$ into a disjoint union
\begin{align}
	\label{eq:dec}
	\cV_n = \cV_n^{(1)} \sqcup \cV_n^{(2)} \sqcup \cV_n^{(3)}.
\end{align}

That is, the first part corresponds to unlabelled $k$-trees with a marked hedron. The second part corresponds to unlabelled $k$-trees with a marked cycle of length at least $2$, such that the cycle-center corresponds a hedron, and the third part to the case where the cycle-center corresponds to a front. We let $V^{(1)}(z)$, $V^{(2)}(z)$, and $V^{(3)}(z)$ denote the corresponding generating series, that is, 
$
V^{(i)}(z) = \sum_{\ell \ge 1} |\cV_\ell^{(i)}| z^\ell.
$
Furthermore, we set $V(z) = \sum_{\ell \ge 1} |\cV_\ell|z^\ell$.

\section{Analysis of the summands}
The generating series and bijective arguments of Section~\eqref{sec:genfun} may be interpreted in terms of $k$-coding trees (rooted for example at black or white vertices) and in terms of $k$-trees (rooted  at a hedron or a front). In order to avoid confusion, we are going to interpret everything in terms of $k$-trees throughout this section. In particular, we regard $B(z)$ as the generating series of unlabelled, uncoloured $k$-trees rooted at a hedron, and $C(z)$ as the generating series of unlabelled, uncoloured $k$-trees that are rooted at a front. A front-colouring of a $k$-tree will always be subject to the restraints stated in Section~\ref{sec:bij}, that the fronts of any hedron are coloured from $1$ to $k+1$ and that fronts that are mirror to each other receive the same colour.

In order to sample an unlabelled $k$-tree with $n$ hedra uniformly at random we may uniformly select a cycle-pointed $k$-tree from $\cV_n$ and then forget about the marked cycle. The decomposition in~\eqref{eq:dec} allows us to divide the study of $\cV_n$ into three cases, depending on the cycle-center. In the following, we treat each part individually.

\subsection{Hedron-rooted $k$-trees}

As unlabelled $k$-trees with a marked hedron correspond bijectively to colour-orbits of unlabelled $k$-coding trees that are rooted at a white vertex, it follows that
\begin{align}
	\label{eq:pp1}
	V^{(1)}(z) = B(z).
\end{align}
That is, $V^{(1)}$-objects correspond bijectively to unlabelled uncoloured hedron-rooted $k$-trees. We are going to make the following observation.
\begin{lemma}
	\label{le:v1}
	Theorems~\ref{te:A}, ~\ref{te:B} and \ref{te:C} hold for the uniformly at random selected unlabelled uncoloured hedron-rooted $k$-tree with $n$ hedra.
\end{lemma}
\begin{proof}
By Lemma~\ref{le:bprob} it follows that in order to uniformly sample an unlabelled hedron-marked $k$-tree, we may first sample a cycle type $\lambda \vdash k+1$ with probability $([z^n]B(z))^{-1} [z^n] B_\lambda(z \rho_k) / z_\lambda$ and then uniformly select a front-coloured $k$-tree with $n$ hedra that is fixed by the permutation~$\pi_\lambda$. By Lemma~\ref{le:analytic} and Equation~\eqref{E:BtoC} the cycle type is exponentially likely to be equal to $1^{k+1}$. The special case  $B_{1^{k+1}}(z) = z C_{1^k}(z)^{k+1}$ of Equation~\eqref{E:BtoC} corresponds to the fact that any $B_{1^{k+1}}$-object may be constructed in a canonical way by gluing the root-fronts of $k+1$ $C_{1^k}$-objects together to form a root-hedron. (See Section~\ref{sec:genfun} for details.) By Lemma~\ref{le:analytic} it holds that $[z^k]C_{1^k}(z) \sim \rho^{-k} k^{-3/2} c_{1^k}$ for some constant $c_{1^k}>0$. Hence either by direct calculations or by applying more general principles of random partitions~\cite{Gi, Gibb} it follows that the largest $C_{1^k}$-component in a random $B_{1^{k+1}}$-object of size $n$ has size $n + O_p(1)$. By Lemma~\ref{le:base} it follows that the limits of Theorem~\ref{te:A} and Theorem~\ref{te:C} hold for random hedron-rooted unlabelled $k$-trees, with the scaling constant of the scaling limit being equal to those for the case of $k$-trees rooted at a front of distinguishable vertices. We may also deduce a tail-bound for the diameter. By Lemma~\ref{le:base} there are constants $C,c>0$ such that probability for the $k$-tree diameter of a $C_{1^k}$-object of size $n$ to be larger than $x\ge0$ is bounded uniformly by $C\exp(-cx^2/n)$. An $n$-sized $B_{1^{k+1}}$-object consists $k+1$ components whose sizes $n_1, \ldots, n_{k+1}$ sum up to $n-1$, since none of them contains the root hedron.  Let $\sB_n$ be a uniformly selected unlabelled, hedron-rooted and front-coloured $k$-tree with $n$ vertices and let $C^1(\sB_n), \ldots, C^{k+1}(\sB_n)$ denote its components. If $\sB_n$ has diameter at least $x$ then at least one of its components has diameter at least $x/2$. Letting $|C^i(\sB_n)|$ denote the number of hedra in the component $C^i(\sB_n)$, it follows that
\begin{align*}
	\Pr{\Di(\sB_n) \ge x} &\le \sum_{n_1 + \ldots + n_{k+1} = n-1} \Pr{|C^{i}(\sB_n)|=n_i,\,i=1 \ldots k+1} \sum_{i=1}^{k+1} C \exp(-cx^2/(4n_i)) \\
	&\le C(k+1) \exp(-c x^2 /(4n)).
\end{align*}
We argued above that the total variational distance of uniform random coloured 
 and uniform random uncoloured 
  hedron rooted $k$-trees (that is, $B_{1^{k+1}}$ and $B$ objects) with $n$ vertices is exponentially small  (as the partition type we considered is exponentially likely to be equal to $1^{k+1}$). It follows that the diameter tail-bound of Theorem~\ref{te:B} holds for the random unlabelled uncoloured hedron-rooted $k$-tree with $n$ vertices. This concludes the proof.
 \end{proof}

\subsection{Cycle-pointed $k$-trees with a hedron cycle-center}
In this section we show that there are only few cycle-pointed $k$-trees with a hedron cycle-center.

\begin{lemma}
\label{le:v2}
A uniformly selected cycle-pointed $k$-tree from the class $\cV_n$ is exponentially unlikely to have a hedron as cycle-center.
\end{lemma}
\begin{proof}
	We have to verify that $|\cV_n^{(2)}| /  |\cV_n| \le C \exp(-cn)$ for some constants $C,c>0$ that do not depend on~$n$. 
	
	For this it suffices to show that the radius of convergence of the generating series $V^{(2)}(z)$ is strictly larger than the radius of convergence $\rho_k$ of the generating series $V(z)$. Indeed, if this is the case, then there is an $\epsilon>0$ such that $V^{(2)}(\rho_k + \epsilon)< \infty$ and hence $|\cV_n^{(2)}|(\rho_k + \epsilon)^n \to 0$. As $|\cV_n| = n [z^n]U(z) \sim a_k n^{-3/2} \rho_k^{-n}$ for some fixed $a_k >0$ by Lemma~\ref{le:analytic} (or Equation~\eqref{eq:recover} below), we know that $|\cV_n|(\rho_k + \epsilon/2)^n \to \infty$. So $|\cV_n^{(2)}| /  |\cV_n| = o(1)\left( \frac{\rho_k + \epsilon/2}{\rho_k + \epsilon} \right)^n$ tends exponentially fast to zero as $n$ becomes large.
	
	Let $V \in \cV_n^{(2)}$ be a cycle-pointed unlabelled uncoloured $k$-tree whose cycle-center is a hedron. Then there is a labelled, uncoloured $k$-tree $K$ with an automorphism $\sigma$ and a marked cycle $c$ of $\sigma$ such that $(K,c)$ looks up to relabelling like $V$.  We may view $K$ as rooted at the cycle-center hedron. Hence $K$ consists of a root hedron whose fronts are identified with the root-fronts of $k+1$ front-rooted $k$-trees $C_1, \ldots, C_{k+1}$. If $\sigma$ sends the label of a hedron contained in $C_i$ to the label of a hedron contained in $C_j$, then it already holds that the restriction of $\sigma$ to the label set of $C_i$ is  an isomorphism from $C_i$ to $C_j$. As the cycle center is a hedron, it follows that there are branches $C_{i_1}, \ldots, C_{i_\ell}$ (each having at least $1$ hedron) with $\ell \ge 2$ such that $\sigma$ cyclically permutes the label sets of the branches. That is, $\sigma$ induces an isomorphism from $C_{i_j}$ to $C_{i_{j+1}}$ if $1 \le j < \ell$ and to $C_{i_1}$ if $j = \ell$. Let $K^\text{col}$ denote any fixed front-coloured version of $K$ (such that fronts of any hedron are coloured from $1$ to $k+1$ and fronts that are mirror to each other receive the same colour.) The automorphism $\sigma$ is not required to respect the colouring, but we know that when we relabel the fronts of $K^\text{col}$ according to $\sigma$ then the result $\sigma.K^\text{col}$ must be some coloured version of $K$. Hence there is a bijection $\lambda \in \mathfrak{S}_{k+1}$ such that $\sigma.K^\text{col}$ equals the recoloured version $\lambda.K^\text{col}$ of $K^\text{col}$. Let $A \in [n]$ be the label of the cycle center hedron in $K^\text{col}$ and let $B \in [n]$ be the label of some hedron of $C_{i_j}$ that contains its root front. Then $\lambda$ must send the colour $a_{i_j} \in [k+1]$ of the unique front contained in the hedra (corresponding to) $A$ and $B$ to the colour of the unique front contained  in the hedra (corresponding to) $\sigma(A)=A$ and $\sigma(B)$. Thus $(a_{i_1}, \ldots, a_{i_\ell})$ is one of the disjoint cyclic factors of the permutation $\lambda$. As $\ell \ge 2$ this implies that $\lambda$ does not have cycle type $1^{k+1}$.
	
	Let $K^\text{unl,col}$ denote the result of dropping the labels of $K^\text{col}$ but retaining the colours. We know that $\lambda.K^\text{col}$ is a relabelled version of $K^\text{col}$, so $K^\text{unl,col}$ is invariant under recolouring according to $\lambda$.  Thus $V$ is formed by dropping the colours of  $K^\text{unl,col}$ and cycle-pointing it in one of the at most $n$ ways such that the cycle center is the root hedron. This shows that
	 \begin{align}
	 	|\cV_n^{(2)}| \le n [z^n] \sum_{\lambda \vdash k+1, \lambda \ne 1^{k+1}} B_\lambda(z).
	 \end{align}
	 By Lemma~\ref{le:analytic} and Equation~\eqref{E:BtoC} it follows that the  series $\sum_{n \ge 1}n z^n [z^n] \sum_{\lambda \vdash k+1, \lambda \ne 1^{k+1}} B_\lambda(z)$ has radius of convergence strictly larger than $\rho_k$. This concludes the proof.
\end{proof}

\subsection{Cycle-pointed $k$-trees with a front cycle-center}

\begin{lemma}
	\label{le:v3}
	Theorems~\ref{te:A}, ~\ref{te:B} and \ref{te:C} hold for the uniformly symmetrically cycle pointed $k$-tree from the class $\cV_n^{(3)}$.
\end{lemma}
\begin{proof}
	Let $V \in \cV_n^{(3)}$ be a cycle-pointed unlabelled uncoloured $k$-tree whose cycle-center is a front. There is a labelled, uncoloured $k$-tree $K$ together with an automorphism $\sigma$ and a marked cycle $c$ of $\sigma$ such that $V$ is the unlabelled version of $(K,c)$.  We consider $K$ as rooted at the cycle-center front. Hence $K$ consists of a set of front-rooted, labelled, uncoloured $k$-trees where the root-front is contained in a unique hedron. We may partition these  into branches that contain hedra of the marked cycle $c$ and branches that do not. Thus $K$ actually consists of two front-rooted components $C^*$ and $C$ that are glued together at their root-fronts, with $C^*$ the subgraph induced by all the branches containing hedra of the cycle $c$ and $C$ the subgraph induced by the branches that do not. We call $C^*$ the marked part and $C$ the unmarked part. If the automorphism $\sigma$ sends the label of a hedron contained in a branch $C_1$ of $K$ to a  label of a hedron contained in another branch $C_2$, then the restriction of $\sigma$ to the label set of $C_1$ is an isomorphism from $C_1$ to $C_2$. In particular, $\sigma$ may be restricted to an automorphism of $C^*$ and consequently also to an automorphism of $C$.
	
	Let us fix a version $K^\text{col}$ of $K$ that is properly front-coloured such that fronts of any hedron are coloured from $1$ to $k+1$ and fronts that are mirror to each other receive the same colour. We additionally that the root-front has colour $k+1$. We know that the result of relabelling  $K^\text{col}$ according to $\sigma$ is a coloured version of $K$, where the root-front still has colour $k+1$. Consequently, there is a bijection $\mu \in \mathfrak{S}_{k}$ such that $\sigma.K^\text{col}$ equals the recoloured version $\mu.K^\text{col}$ of $K^\text{col}$. Let $\cV^{(3), \text{sym}}_n \subset \cV_n^{(3)}$ be the subset of all cycle-pointed $k$-trees where $\sigma$ and $\mu$ may \emph{not} be chosen in such a way that $\mu$ is the identity map. (This does not depend on the choice of $K^\text{col}$.) Let $K^{\text{unl,col}}$ denote the unlabelled, coloured $k$-tree obtained by dropping the labels but retaining the colours of~$K^\text{col}$. As $\sigma.K^\text{col} = \mu.K^\text{col}$ it follows that $K^{\text{unl,col}}$ is invariant under recolouring by $\mu$. Since any unlabelled $k$-tree with $n$ hedra has at most $n$ unlabelled cycle-pointed versions where the cycle-center is a front, it follows that
	\begin{align}
		\label{eq:asdf}
		|\cV^{(3), \text{sym}}_n| \le n [z^n] \sum_{\mu \vdash k, \mu \ne 1^k} C_\mu(z).
	\end{align}
	 By Lemma~\ref{le:analytic}  we know that the series $\sum_{n \ge 1} n [z^n] \sum_{\mu \vdash k, \mu \ne 1^k} C_\mu(z)$ has radius of convergence strictly larger than $\rho_k$. This implies that a uniformly selected element of $\cV^{(3)}_n$ is exponentially unlikely to lie in $\cV^{(3), \text{sym}}_n$. We set $\cV_n^{(3), \text{dec}} := \cV^{(3)}_n \setminus \cV^{(3), \text{sym}}_n$, so that
	 \begin{align}
	 	\label{eq:pp2}
	 	V^{(3)}(z) = V^{(3), \text{sym}}(z) + V^{(3), \text{dec}}(z)
	 \end{align}
 	 with $ V^{(3), \text{sym}}(z) =   \sum_{n \ge 1} |\cV^{(3), \text{sym}}_n| z^n$ and $ V^{(3), \text{dec}}(z) = \sum_{n \ge 1} |\cV^{(3), \text{dec}}_n | z^n $.

 	 \begin{figure}[t]
 	 	\begin{center}
 	 		\includegraphics[scale=1.0]{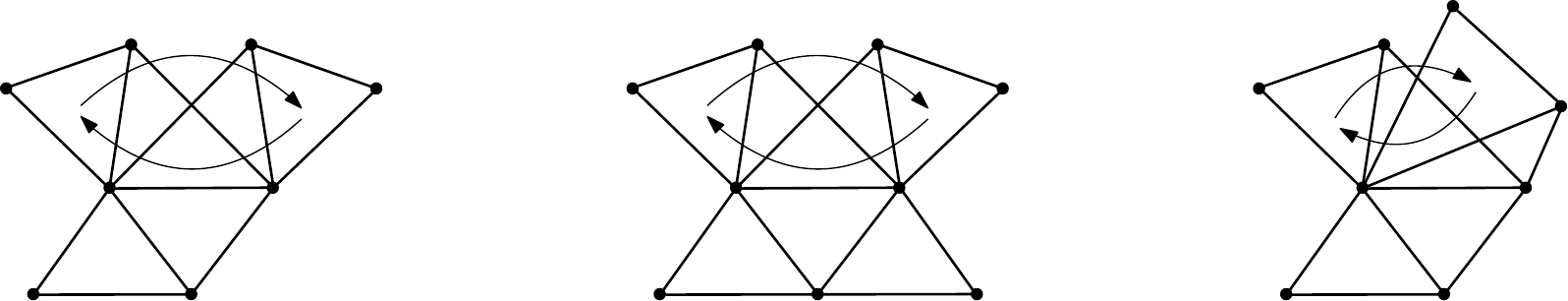}
 	 	\end{center}
 	 	\caption{The $2$-tree on the left is \emph{not} cycle-pointed, since no automorphism has the marked cycle as one of its disjoint cyclic factors. The other two are cycle-pointed with a front as cycle center. The middle one belongs to the  class $\cV_n^{(3), \text{sym}}$ for $n=7$ \label{F:5}} and the one on the right to the class $\cV_n^{(3), \text{dec}}$.
 	 \end{figure}
 	 
 	 Suppose that the unlabelled uncoloured $k$-tree $V$ we considered above lies in the set $\cV_n^{(3), \text{dec}}$.  As $V \in \cV_n^{(3), \text{dec}}$ we may assume that $\sigma$ got chosen in a way that preserves the colouring of the coloured version $K^\text{col}$, that is $\sigma.K^{\text{col}} = K^\text{col}$. We argued above that the automorphism $\sigma$ restricts to an automorphism of the marked part $C^*$ (and  of the remaining part $C$). That is, at least one (and hence all) colourings of the marked part admit a colour-preserving automorphism having the marked cycle as one of its disjoint cyclic factors.  This is a key observation. Arbitrary elements of $\cV_n^{(3)}$ may have a marked part whose marked cycle may only be extended to automorphisms involving some form of rotation of the root hedron (that is, they are not colour preserving) such as the middle $k$-tree of Figure~\ref{F:5}. This imposes symmetry constraints (that is,  invariance under non-trivial recolouring) on the unmarked part. For this reason we could show in~\eqref{eq:asdf} that there are much less elements in $\cV_n^{(3), \text{sym}}$ than in $\cV_n^{(3)}$. For elements of $\cV_n^{(3), \text{dec}}$ there are no such symmetry constraints. If the marked cycle may be extended to an automorphism of the marked part $C^*$ that preserves a front-colouring, then $C$ may be equal to any front-rooted unlabelled $k$-tree such that the total number of hedra of $C$ and $C^*$ sums up to $n$. (The abbreviation ``{dec}'' for ``decoupled'' intends to indicate this.) In fact, we may always choose $\sigma$ in such a way that it pointwisely fixes all hedra of $C$. Note that, given a marked and an unmarked part, there is in general no canonical way to glue them together at the root-front. We will get to this in a moment.

 	 Let us first examine the constraints on the marked part.  Since $\sigma$ preserves the colouring, this means that if we distinguish the vertices of the root-front  of $C^*$ by ordering them linearly, then  $C^*$ consists of identical branches glued together in the unique way according to the order on the root-front. See for example the $k$-tree on the right of Figure~\ref{F:5}. The automorphism $\sigma$ cyclically permutes the label sets of the branches of $C^*$. Let $C_1, \ldots, C_\ell$, $\ell \ge 2$ denote the branches of $C^*$ such that $C_i$ gets sent to $C_{i+1}$ by $\sigma$ if $i < \ell$, an to $C_1$ if $i = \ell$. The disjoint cyclic factor of $\sigma$ that corresponds to the marked cycle must be of the form
 	 \[
 	 (a_{1,1}, \ldots, a_{1,\ell}, a_{2,1}, \ldots, a_{2,\ell}, \quad \ldots \quad, a_{r,1}, \ldots, a_{r,\ell})
 	 \]
 	 for some $r \ge 1$ such that for each $1 \le j \le \ell$ the labels $a_{1,j}, \ldots, a_{r,j}$ correspond to distinct hedra of the branch~$C_j$. Note that the restriction of the power $\sigma^\ell$ to the label set of $C_j$ is an automorphism of $C_j$, and $(a_{j,1}, \ldots, a_{j,\ell})$ is one of its disjoint cyclic factors. Hence, up to hedron labels, $C^*$ is completely described by the number $\ell \ge 2$ of branches together with a single cycle pointed branch $(C_j, (a_{1,j}, \ldots, a_{r,j}))$. See Figure~\ref{F:8} for an illustration of how to construct a marked part in this way. Note that not every marked part constructed in this way is admissible, in particular the example of Figure~\ref{F:8} admits no front-colour preserving automorphism having the marked cycle as one of its disjoint factors. This is the case if and only if the cycle-pointed branch has this property.
 	 
 	 Let $\bar{B}(z)$ be the power series so that $[z^n]\bar{B}(z)$ counts the number of front-rooted unlabelled uncoloured $k$-trees with $n$ hedra where the root-front is contained in a unique hedron. Let
 	 $
 	  \bar{B}^{\circ_\text{w}}(z)
 	 $
 	 count unlabelled uncoloured cycle-pointed branches that admit an automorphism that has the marked cycle as one of its disjoint factors and preserves a given (and hence all) front-colourings. (The ``w'' indicates that they are ``well'' pointed.) The generating series $M(z)$ of the  class $\cM$ of marked parts that are admissible for elements of the class $\cV^{(3), \text{dec}} = \bigcup_{n \ge 1} \cV^{(3), \text{dec}}_n$ is consequently given by
 	 \[
 	 	M(z) = \sum_{\ell \ge 2}\bar{B}^{\circ_\text{w}}(z^\ell),
 	 \]   with the variable $z$ indexing the number of hedra.
 	 
 	 As mentioned before, there may be various ways to glue an unlabelled uncoloured marked part and an unmarked front-rooted unlabelled uncoloured $k$-tree together at the root-front. In order to handle this we introduce colours. Consider the set of all front-colourings of $\cM$-objects such that the root-front receives colour $k+1$.  Note that each branch in a coloured $\cM$-object is coloured identically, as fronts that are mirror to each other receive the same colour. Hence a coloured $\cM$-object is constructed out of copies of a single coloured $\bar{B}^{\circ_\text{w}}$-object. For each cycle type $\mu \vdash k$ let  $(\bar{B}^{\circ_\text{w}})_\mu(z)$ denote the generating series of the class of all colourings of $\bar{B}^{\circ_\text{w}}$ objects such that the root-front receives colour $k+1$.  Likewise, we let $M_\mu(z)$ count $\mu$-invariant coloured $\cM$-objects. Then
 	 \[
 	 	M_\mu(z) = \sum_{\ell \ge 2} (\bar{B}^{\circ_\text{w}})_\mu(z^\ell).
 	 \]
 	 It's easy to see that $(\bar{B}^{\circ_\text{w}})_\mu(z)$ and $\bar{B}_\mu(z)$ have the same radius of convergence.  By Lemma~\ref{le:analytic} it follows that~$M_\mu(z)$ has radius of convergence strictly larger than $\rho_k$. Now, consider the class of colourings of $\cV^{(3), \text{dec}}$-objects where again the root-front is required to receive colour $k+1$. Applying Burnside's Lemma~\ref{le:burn} yields
 	 \begin{align}
 	 \label{eq:pp3}
 	 V^{(3), \text{dec}}(z) = \sum_{\mu \vdash k} z_\mu^{-1} C_\mu(z) \sum_{\ell \ge 2}(\bar{B}^{\circ_\text{w}})_\mu(z^\ell).
	\end{align}
	Using Lemma~\ref{le:analytic} it follows that all summands with $\mu \ne 1^k$ have radius of convergence strictly larger than $\rho_k$. The summand for $\mu=1^k$ represents pairs of a $\bar{B}_{1^k}$-object (that bijectively corresponds to an unlabelled uncoloured $k$-tree rooted at a front of \emph{distinguishable} vertices, see Section~\ref{sec:disting}) and an $M_{1^k}$-object, that are glued together in a \emph{canonical} way. It follows from Lemma~\ref{le:bprob} that there a constants $C,c>0$ such that the  total variation distance between the uniform measure on $\cV_n^{(3)}$ and the uniform measure on $n$-hedron unlabelled, uncoloured, cycle-pointed $k$-trees obtained from  ``$\bar{B}_{1^k}$, $M_{1^k}$''-pairs is bounded by $C \exp(-cn)$ for all $n$. As $M_{1^k}(z)$ has radius of convergence strictly larger than $\rho_k$, it follows easily from the asymptotic expansion of $[z^n]\bar{B}_{1^k}(z)$ that the marked part as asymptotically bounded size. Hence it follows from Lemma~\ref{le:base} that Theorems~\ref{te:A}, ~\ref{te:B} and \ref{te:C} hold for the uniformly symmetrically cycle pointed $k$-tree from the class $\cV_n^{(3)}$.
\end{proof}

\begin{figure}[t]
	\begin{center}
		\includegraphics[scale=1.5]{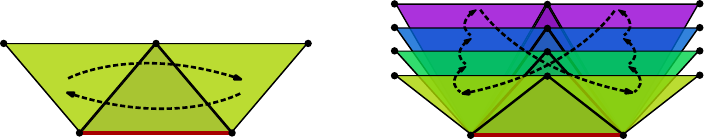}
	\end{center}
	\caption{Constructing a marked part with a cycle of length $8$ out of $4$ copies of a single marked branch with a cycle of length $2$.}
	\label{F:8}
\end{figure}

\subsection{Conclusion}
\label{sec:conclusion}
Theorems~\ref{te:A},~\ref{te:B} and \ref{te:C} follow easily from the decomposition~\eqref{eq:dec} and Lemmas~\ref{le:v1},~\ref{le:v2} and \ref{le:v3}. The generating series we derived in Equations~\eqref{eq:pp1}, \eqref{eq:pp2}, and \eqref{eq:pp3} may be summarized by 
\begin{align}
\label{eq:sum}
zU'(z) = B(z)  + V^{(2)}(z) + V^{(3), \text{sym}}(z) + \sum_{\mu \vdash k} z_\mu^{-1} {C}_\mu(z) \sum_{\ell \ge 2}  (\bar{B}^{\circ_\text{w}})_\mu(z^\ell).
\end{align}
This may be seen as a substraction-free alternative to the dissymmetry equation~\cite[Lem. 6]{GeGai}. The series $V^{(2)}(z)$, $V^{(3), \text{sym}}(z)$ and $\sum_{\mu \vdash k, \mu \ne 1^k} z_\mu^{-1} \bar{B}_\mu(z) \sum_{\ell \ge 2}  (\bar{B}^{\circ_\text{w}})_\mu(z^\ell)$ have radius of convergence strictly larger than~$\rho_k$, and $\bar{B}_{1^k}(z)$ has as dominant singularity of square-root type at $\rho_k$. We may apply Equations~\eqref{E:BbarC1} and \eqref{eq:recur} together with general principles (for example~\cite[Lem. 3.2]{Gibb}) to deduce that
\begin{align}
n[z^n]U(z) \sim \frac{(k \rho_k)^{-k}}{k k!}\left(1 + k\sum_{\ell \ge 2} (\bar{B}^{\circ_w})_{1^k}(\rho_k^\ell)\right)[z^n] \bar{B}_{1^{k+1}}(z).
\end{align}
In general the operations of cycle-pointing and colouring $k$-trees do not commute. However, colouring ``well'' pointed $k$-trees is the same as cycle-pointing coloured $k$-trees. Hence
\begin{align}
(\bar{B}^{\circ_w})_{1^k}(z) = z \bar{B}'_{1^k}(z).
\end{align}
There is an asymptotic expansion
\begin{align}
[z^n]\bar{B}_{1^k}(z) \sim \sqrt{\frac{1 + k \sum_{\ell \ge 2} \bar{B}'_{1^k}(\rho_k^\ell)\rho_k^\ell }{2 \pi k^2}} n^{-3/2} \rho_k^{-n}
\end{align} that may be deduced from Equation~\eqref{eq:recur} using \cite[Thm. 28]{belletal}. We have thus recovered the asymptotic expansion 
\begin{align}
\label{eq:recover}
	[z^n]U(z) \sim \frac{(k \rho_k)^{-k}}{k^2 k! \sqrt{2\pi}}\left(1 + k\sum_{\ell \ge 2} (\bar{B}^{\circ_w})_{1^k}(\rho_k^\ell)\right)^{3/2} n^{-5/2} \rho_k^{-n}
\end{align}
that was proven in~\cite[Thm. 3]{DJ:14} via the dissymmetry equation.


\begin{thebibliography}{}
	
	
\bibitem{MR1166406}
D.~Aldous, {\em The continuum random tree. {II}. {A}n overview.}, Stochastic analysis ({D}urham, 1990), volume 167 of London Math. Soc. Lecture Note Ser., pages 23--70. Cambridge Univ. Press, Cambridge, (1991).
	
\bibitem{belletal}
J. P. Bell, S. N. Burris, and K. A. Yeats. Counting rooted trees: the universal law $t(n) \sim C \rho^{-n} n^{−3/2}$. Electron.
J. Combin., 13(1):Research Paper 63, 64 pp. (electronic), 2006.

\bibitem{BFKV:07}
M. Bodirsky, \'{E}. Fusy, M. Kang and S. Vigerske, {\em Boltzmann samplers, P\'{o}lya theory, and cycle pointing}, SIAM J. Comput., 40(3), 721–769 (2011).

\bibitem{BLL}
F. Bergeron, G. Labelle and P. Leroux, {\em Combinatorial species and tree-like structures}, Encyclopedia of Mathematics and its Applications, vol. 67, Cambridge University Press, Cambridge, 1998, Translated from the 1994 French original by Margaret Readdy, with a foreword by Gian-Carlo Rota.


\bibitem{BP:69}
L. W. Beineke and R. E. Pippert, {\em The number of labeled
k-dimensional trees}, J. Combin. Theory, 6, 200-205 (1969).

\bibitem{Dar}
A. Darrasse and M. Soria {\em Limiting Distribution for Distances in k-Trees} Fiala J., Kratochvíl J., Miller M. (eds) Combinatorial Algorithms. IWOCA 2009. Lecture Notes in Computer Science, vol 5874. Springer, Berlin, Heidelberg, (2009).


\bibitem{DJ:14}
M. Drmota and E.Y. Jin, {\em An asymptotic analysis of labeled and unlabeled $k$-trees}, Algorithmica, 75(4), 579–605 (2014).


\bibitem{DJS:16}
M. Drmota, E.Y. Jin and B. Stufler {\em Graph limits of random graphs from a subset of connected $k$-trees}, ArXiv e-prints, (2016)


\bibitem{Flajolet}
P. Flajolet and R. Sedgewick, {\em Analytic combinatorics},
Cambridge University Press (2010).

\bibitem{Foata}
D. Foata, {\em Enumerating k-trees}, Discrete Math. 1, 181-186
(1971).

\bibitem{Fowler:02}
T. Fowler, I. Gessel, G. Labelle, and P. Leroux, {\em The
specification of 2-trees}, Adv. Appl. Math. 28, 145-168 (2002).

\bibitem{Gai}
A. Gainer-Dewar, {\em $\Gamma$-species and the enumeration of
$k$-trees}, Electron. J. Combin. 19(4), P45 (2012).

\bibitem{GeGai}
A. Gainer-Dewar and I. M. Gessel, {\em Counting unlabeled
$k$-trees}, Journal of Combinatorial Theory A, 126, pp.~177-193, (2014).


\bibitem{HP:68}
F. Harary and E. M. Palmer, {\em On acyclic simplicial complexes},
Mathematika 15, 115-122 (1968).

\bibitem{HP:73}
F. Harary and E. M. Palmer, {\em Graphical Enumeration}, Academic
Press, New York-London (1973).

\bibitem{I:15}
A.D. Iriza, {\em Enumeration and random generation of unlabeled classes of graphs: a practical study of cycle-pointing and the dissymmetry theorem}, master thesis, Princeton University, (2015).

\bibitem{J1}
A. Joyal, {\em Une th\'{e}orie combinatoire des s\'{e}ries formelles}, Adv. in Math. 42, 1-82 (1981).

\bibitem{Mi}
G. Miermont, {Tessellations of random maps of arbitrary genus},
Ann. Sci. Éc. Norm. Supér. 42, fascicule 5, 725--781 (2009). 

\bibitem{Moon}
J. W. Moon, {\em The number of labeled k-trees}, J. Combin. Theory
Ser. A 6, 196-199 (1969).

\bibitem{Ot}
Richard Otter, {\em The number of trees}, Ann. of Math. (2), 49:583--599, (1948).

\bibitem{Po}
G. P\'{o}lya, {\em Kombinatorische Anzahlbestimmungen f\"{u}r Gruppen, Graphen und chemische Verbindungen}, Acta Mathematica 68(1), 145-254 (1937).



\bibitem{Gi}
{\sc B.~{Stufler}}, {\em Gibbs partitions: the convergent case}, To appear in Random Structures \& Algorithms.


\bibitem{Streelike}
{\sc B.~{Stufler}}, {\em Limits of random tree-like discrete structures}, ArXiv e-prints, (2016)


\bibitem{SENR}
{\sc B.~{Stufler}}, {\em {Random enriched trees with applications to random graphs}}, ArXiv e-prints,  (2015).
	


\bibitem{outer}
{\sc B.~{Stufler}}, {\em {Scaling limits of random outerplanar maps with independent link-weights}}, Ann. Inst. H. Poincaré Probab. Statist. Volume 53, Number 2, 900-915, (2017).
	
\bibitem{2014arXiv1412.6333S}
{\sc B.~{Stufler}}, {\em {The continuum random tree is the scaling limit of
		unlabelled unrooted trees}}, ArXiv e-prints,  (2014).


\bibitem{Gibb}
{\sc B.~{Stufler}}, {\em Unlabelled Gibbs partitions}, ArXiv e-prints,  (2016).



\end{thebibliography}
\end{document}